\documentclass[12 pt]{amsart}
\usepackage{amsmath,amssymb,amsthm}
\usepackage{longtable}
\usepackage{a4wide}
\usepackage{color}
\usepackage{enumerate}
\usepackage{cleveref}
\usepackage{graphicx}

\newtheorem{thm}{Theorem}[section]
\newtheorem*{thm*}{Theorem}
\newtheorem*{conj*}{Conjecture}

\newtheorem{lem}[thm]{Lemma}
\newtheorem{prop}[thm]{Proposition}

\theoremstyle{remark}
\newtheorem{rem}[thm]{Remark}

\theoremstyle{definition}

\newtheorem{defn}[thm]{Definition}

\DeclareMathOperator{\Alt}{Alt}
\DeclareMathOperator{\Sym}{Sym}
\DeclareMathOperator{\Gal}{Gal}
\DeclareMathOperator{\Aut}{Aut}

\DeclareMathOperator{\Cay}{Cay}

\DeclareMathOperator{\Z}{Z}

\newcommand{\A}{\mathrm{A}}
\newcommand{\B}{\mathrm{B}}
\newcommand{\C}{\mathrm{C}}
\newcommand{\D}{\mathrm{D}}
\newcommand{\E}{\mathrm{E}}
\newcommand{\F}{\mathrm{F}}
\newcommand{\G}{\mathrm{G}}
\newcommand{\N}{\mathrm{N}}

\newcommand{\OO}{\mathrm{O}}

\newcommand{\FF}{\mathbb{F}}

\newcommand{\PSU}{\mathrm{PSU}}
\newcommand{\SU}{\mathrm{SU}}
\newcommand{\GL}{\mathrm{GL}}
\newcommand{\PSL}{\mathrm{PSL}}
\newcommand{\SL}{\mathrm{SL}}
\newcommand{\Sp}{\mathrm{Sp}}

\renewcommand{\geq}{\geqslant}
\renewcommand{\leq}{\leqslant}

\title[Finite simple groups that are CCA]{A finite simple group is CCA if and only if it has no element of order four}

\author{Luke Morgan}
\author{Joy Morris}
\author{Gabriel Verret}

\address{Luke Morgan and Gabriel Verret$^*$\\
Centre for the Mathematics of Symmetry and Computation, School of Mathematics and Statistics (M019)\\
The University of Western Australia\\
Crawley, 6009\\
Australia} 
\email{luke.morgan@uwa.edu.au}

\address{Joy Morris\\
Department of Mathematics and Computer Science\\
University of Lethbridge\\
Lethbridge, AB T1K 3M4\\
Canada}
\email{joy.morris@uleth.ca}

\address{$*$ Current address: Department of Mathematics, The University of Auckland, Private Bag 92019, Auckland 1142, New Zealand.}
\email{g.verret@auckland.ac.nz}	

\thanks{This research was supported by the Australian Research Council grants DE130101001 and DE160100081, by the Natural Science and Engineering Research Council of Canada, and by the Cheryl E. Praeger Visiting Research Fellowship from The University of Western Australia.}

\begin{document}

\begin{abstract}
A Cayley graph for a group $G$ is \emph{CCA} 
if every automorphism of the graph that preserves the edge-orbits under the regular representation of $G$ 
 is an element of the normaliser of $G$.  A group $G$ is then said to be  \emph{CCA} if every  connected Cayley graph  on $G$ is CCA.
We show that a finite simple group is CCA if and only if it has no element  of order 4. We also show that ``many'' $2$-groups 
are non-CCA. 
\end{abstract}

\maketitle

\section{Introduction}
All groups and all graphs in this paper are finite. Let $G$ be a group and let $S$ be an inverse-closed subset of $G$. The \emph{Cayley graph} $\Cay(G,S)$ of $G$ with respect to $S$ is the  graph  with vertex-set $G$ and, for every $g\in G$ and $s\in S$, an edge $\{g,sg\}$. This graph admits a natural  edge-colouring
in which an edge $\{g,sg\}$ is coloured $\{s,s^{-1}\}$. The colour-preserving automorphism group is denoted $\Aut_c(\Cay(G,S))$ and we define $\Aut_{\pm 1}(G,S)=\{\alpha \in \Aut(G) \mid s^\alpha \in \{s,s^{-1}\}$ for all $s \in S \}$. It is easy to see that $G_R \rtimes\Aut_{\pm 1}(G,S)\leq \Aut_c(\Cay(G,S))$, where $G_R$ is the right-regular representation of $G$ and, in fact, the former group is  precisely the normaliser of $G_R$ in $\Aut_c(\Cay(G,S))$.

\begin{defn}[\cite{FirstCCA}]
The Cayley graph $\Cay(G,S)$ is \emph{CCA} (Cayley colour automorphism) if $\Aut_c(\Cay(G,S))=G_R \rtimes\Aut_{\pm 1}(G,S)$. The group $G$ is  \emph{CCA} if every connected Cayley graph on $G$ is CCA.
\end{defn}

Thus, $\Cay(G,S)$ is CCA if and only if $G_R$ is normal in $\Aut_c(\Cay(G,S))$, c.f. \cite[Remark 6.2]{FirstCCA}. Note that $\Cay(G,S)$ is connected if and only if $S$ generates $G$. 

Previous results on the CCA problem  have focused on groups of odd order and, more generally, on solvable groups (see \cite{CCASquarefree} and \cite{FirstCCA}  for example).
In this paper, we will focus on two ends of the spectrum of groups: non-abelian
simple groups and  
$2$-groups.

In Section~\ref{prelim}, we introduce some basic terminology and previous results on the CCA  problem. In particular, Proposition~\ref{prop:gensets} is a condition from \cite{FirstCCA} that is sufficient to guarantee that a group is non-CCA. This condition requires the group to contain elements of order four. We also include various results that will allow us to apply this condition to many of the groups we study in this paper.

In Section~\ref{sec:Many2groups} we focus on 
$2$-groups. We show that a lower bound on the number of groups of order $2^n$ that are non-CCA has the same leading term as the total number of groups of order $2^n$. 
In Sections~\ref{sec:simple no 4} and ~\ref{sec:SimpleWith}  we 
prove the following.

\begin{thm}
\label{thm:simplegrouptheorem}
 A finite simple group is CCA if and only if it has no element of order four.
\end{thm}

This theorem also holds for almost simple groups whose socle is either an alternating group or a Suzuki group; the proof of this is also 
in Section~\ref{sec:SimpleWith}.

The proof of Theorem~\ref{thm:simplegrouptheorem} involves some case-by-case analysis of  finite simple groups, and relies on their classification.

\section{Preliminaries}\label{prelim}
We begin by describing a sufficient criterion for a group to be non-CCA that appeared in~\cite{FirstCCA}. This criterion is surprisingly powerful, as will be made abundantly clear in Sections~\ref{sec:Many2groups} and~\ref{sec:SimpleWith}. We first need the following definition.

\begin{defn}
\label{defn:non-cca set}
Let $G$ be a group. A \emph{non-CCA triple} of $G$  is a triple $(S,T,\tau)$ where $S$ and $T$ are subsets of $G$ and $\tau$ is an involution in $G$ such that the following hold:
\begin{enumerate}[(i)]
\item[(Ai)] $G= \langle S \cup T \rangle$;
\item[(Aii)] $\tau$ inverts or centralises every element of $S$;
\item[(Aiii)] $t^2 = \tau$ for every $t\in T$;
\item[(Aiv)] $ \langle S \cup \{ \tau \} \rangle \neq G$;
\item[(Av)] either $\tau$ is non-central in $G$ or $|G : \langle S \cup \{ \tau \} \rangle | > 2$.
\end{enumerate}
\end{defn}

We may sometimes abuse notation and write $(S,t,\tau)$ for the non-CCA triple $(S,T,\tau)$ when $T=\{t\}$.

\begin{prop}[{\cite[Proposition 2.5]{FirstCCA}}]
\label{prop:gensets}
If $(S,T,\tau)$ is a non-CCA triple of $G$, then $\Cay(G, S \cup T)$ is connected and non-CCA and thus $G$ is non-CCA.
\end{prop}

For $G$ a group and $\tau$ an involution of $G$, we set 
$$S_G(\tau) = (C_G(\tau) \cup \{ y\tau \mid y\in G\text{ and }y^2=1 \})-\{1\}.$$

\begin{rem}
\label{rem:invs}The subgroup $\langle S_G(\tau)\rangle$ contains every involution of $G$, and thus contains the normal subgroup of $G$ generated by the set of involutions of $G$.
\end{rem}
The following lemmas  prove useful in allowing us to apply Proposition~\ref{prop:gensets}. 

\begin{lem}
\label{lem:stau}
 If $G$ is a group with an involution $\tau$, then
$$S_G(\tau)=\{x \in G \mid x^\tau \in \{ x, x^{-1}\} \}.$$
\end{lem}
\begin{proof}
By definition,  $x^\tau=x$ if and only if $x\in C_G(\tau)$. If $u=y\tau$ with $y^2=1$, then $u^\tau = \tau y\tau \tau = \tau y = \tau^{-1} y^{-1} = (y\tau)^{-1} = u^{-1}$. In the other direction, if
$x^\tau = x^{-1}$, then $(x\tau)^2=1$ hence $x = (x\tau)\tau  \in S_G(\tau)$, as required.
\end{proof}

\begin{rem}
Lemma~\ref{lem:stau}  shows that (Aii) of Definition~\ref{defn:non-cca set} holds whenever we use some $S_G(\tau)$ as the first entry of a putative non-CCA triple, with $\tau$ as the final entry.
This fact will be used repeatedly throughout Section~\ref{sec:SimpleWith},  usually  without explicit reference.
\end{rem}
%
%

We now  state two results on colour-preserving automorphisms of Cayley graphs.

\begin{lem}[{\cite[Lemma~6.3]{FirstCCA}}]\label{VxStab2Group}
The vertex-stabiliser in the colour-preserving group of automorphisms of a connected Cayley graph is a $2$-group.
\end{lem}

\begin{lem}[{\cite[Lemma 2.4]{cca1}}]\label{coolnew}
Let $\Gamma=\Cay(G,S)$, let $A$ be a colour-preserving group of automorphisms of $\Gamma$, let $N$ be a normal $2$-subgroup of $A$ and let $K$ be the kernel of the action of $A$ on the $N$-orbits. If $K_v\neq 1$ for some $v\in \Gamma$, then $S$ contains an element of order four.
\end{lem}


\section{Many $2$-groups are not CCA}\label{sec:Many2groups}

Abelian CCA groups were determined in~\cite[Proposition 4.1]{FirstCCA}. From this classification, it follows that, while the number of abelian CCA groups of order $2^n$ increases with $n$, almost all abelian $2$-groups are non-CCA.
We are not able to prove a result quite this strong for all $2$-groups but, using a slightly modified version of an argument of Higman, we get the following.

\begin{thm}\label{Many2Groups}
There are at least $2^{\frac{2}{27}n^3+O(n^2)}$ pairwise non-isomorphic groups of order $2^n$ that are non-CCA.
\end{thm}
\begin{proof}
We assume that $n\geq 3$, and follow the account by Sims~\cite[pg.151--152]{Sims} of a result of Higman~\cite[Theorem~2.1]{Higman}. Let $r$ and $s$ be positive integers such that $r+s=n$. 
  For  $1\leq i\leq r$ and  $1\leq j\leq s$, let $b(i,j) \in \{0,1\}$. For  $1\leq i<j\leq r$ and $1\leq k\leq s$, let  $c(i,j,k) \in \{0,1\}$. The relations 


\medskip\begin{center}
\begin{tabular}{rclr}
 $h_i^2$&$=$&$1$, &  $1\leq i\leq s$,    \\
$[h_i,h_j]$&$=$&$1$, &  $1\leq i\leq j\leq s$,   \\
 $[g_i,h_j]$&$=$&$1$, &  $1\leq i\leq r$, $1\leq j\leq s$,    \\
 $g_i^2$&$=$&$h_1^{b(i,1)}\cdots h_s^{b(i,s)}$, &  $1\leq i\leq r$,    \\
 $[g_i,g_j]$&$=$&$h_1^{c(i,j,1)}\cdots h_s^{c(i,j,s)}$, &  $1\leq i<j\leq r$,    \\
\end{tabular}
\end{center}
\medskip \noindent
on $\{g_1,\ldots,g_r,h_1,\ldots,h_s\}$ define a group of order $2^n$. The number of ways of choosing the $b(i,j)$s and the $c(i,j,k)$s is $2^{{r\choose 2}s+rs}$ which, if we take $r=\lfloor 2n/3\rfloor$, is $2^{2n^3/27+O(n^2)}$.
Moreover, Higman showed that the number of choices of the $b(i,j)$s and the $c(i,j,k)$s which determine isomorphic groups is $2^{O(n^2)}$.

We now add the extra requirement that $g_r^2=g_{r-1}^2=h_1$. This completely determines $b(r,j)$ and $b(r-1,j)$ for every $j\in\{1,\ldots,s\}$. The number of ways of choosing the parameters is now $2^{{r\choose 2}s+(r-2)s}$ which is again $2^{2n^3/27+O(n^2)}$ for $r=\lfloor 2n/3\rfloor$ and, by Higman's result, we still get $2^{2n^3/27+O(n^2)}$ pairwise non-isomorphic groups.

Let $G=\langle g_1,\ldots,g_r,h_1,\ldots,h_s\rangle$ be such a group.   Let $S=\{g_1,\ldots,g_{r-2},h_1,\ldots,h_s\}$, let $\tau=h_1$ and let $T=\{g_{r-1},g_r\}$. We show that $(S,T,\tau)$ is a non-CCA triple and thus $G$ is not CCA by Proposition~\ref{prop:gensets}. Clearly, $G= \langle S \cup T \rangle$. Moreover, $\tau$ is central in $G$ and $g_{r-1}^2=g_r^2=\tau$.  Finally, let $X=\langle S\cup \{\tau\}\rangle$ and let $H=\langle h_1,\ldots,h_s\rangle$. Note that $H\leq X$ and that $G/H$ is an elementary abelian $2$-group of order $2^r$ with $\{Hg_1,\ldots, Hg_r\}$ forming a basis. This implies that $|G:X|=4$ and thus $(S,T,\tau)$ is a non-CCA triple.
\end{proof}

\begin{rem}
Note that, by~\cite{Sims},  the  number of groups of order $2^n$ is  $2^{\frac{2}{27}n^3+O(n^{8/3})}$. Still, Theorem~\ref{Many2Groups} falls short of proving that almost all $2$-groups are non-CCA, although this seems likely to be the case.
\end{rem}

\section{Simple groups with no element of order four}\label{sec:simple no 4}

In this section, we show that simple groups with no element of order four are CCA.  It is easy to see that cyclic groups of prime order are CCA. (See \cite[Proposition 4.1]{FirstCCA} or  ~\cite{Joy-CCA}. This can also be seen as a  consequence of Burnside's Theorem~\cite[Theorem 3.5A]{Dixon-Mort}.) We therefore restrict our attention to non-abelian simple groups with no element of order four. Such groups were classified by Walter.




\begin{thm}[\cite{Walter}]\label{WalterTheorem}
A non-abelian simple group has no element of order four if and only if it is isomorphic to one of the following:
\begin{itemize}
\item $\PSL(2,2^e)$, $e\geq 2$,
\item $\PSL(2,q)$, $q \equiv\pm3\pmod 8$, $q\geq 5$,
\item a Ree group $^2\G_2(3^{2n+1})$, $n\geq 1$,
\item the Janko group $J_1$.
\end{itemize}
\end{thm}

We will need the following result concerning the dimensions of irreducible $\FF_2$-modules for $\PSL(2,2^e)$.


\begin{lem}
\label{lem:sl2mods}
Let $G=\PSL(2,2^e)$ and let $W$ be a non-trivial irreducible $\FF_2G$-module. Then $\dim_{\FF_2}(W) \geqslant 2e$.
\end{lem}
\begin{proof}
Let $K=\FF_2$ and set $L=\FF_{2^f}$ where $f$ is minimal such that  $W^L:= L \otimes_{K}W$  is a sum of absolutely irreducible $LG$-modules.  Note that $f$ divides $e$  since $\FF_{2^e}$ is a splitting field for $G$. Using \cite[(26.2)]{asch} and the notation from loc.~cit.~we have
$$W^L = \bigoplus_{i=1}^a V_i$$
where each $V_i$ is a Galois twist of $V:=V_1$ and  $a=|\Gamma : \N_\Gamma(V)|$ with $\Gamma=\Gal(L,K)$. Since $V$ cannot be written over a subfield of $L$,  \cite[(26.5)]{asch} implies that $\N_\Gamma(V)=1$ and so $a=f$.

We now use the Brauer-Nesbitt Theorem as formulated in \cite[Section 5.3]{colva} and borrow the notation   established there. Since $V$ is  irreducible,   \cite[Theorem 5.3.2]{colva} states that  $V=M(n)$ for some integer $n$ with $0 \leqslant n \leqslant 2^e-1$. Further, since $V$ is written over $L$, there are $0\leqslant a_0,\dots, a_{e-1} \leqslant 1$ such that $n=\sum_{i=0}^{e-1} a_i 2^i$ and we have $V=M(a_0)\otimes M(a_1)^{\phi} \otimes \dots \otimes M(a_{e-1})^{\phi^{e-1}}$ and $\dim_L(V)=(a_0+1)(a_1+1) \cdots (a_{e-1}+1)$. Since $V$ is non-trivial we have $n \geqslant 1$, so there is some $i$ such that $a_i=1$. Now  since $V$ is written over $L$, \cite[Corollary 5.3.3]{colva} gives that $a_i=a_j$ if $i \equiv j \mod f$. Hence  $\dim_L(V) \geqslant 2^{\frac{e}{f}}$ and we obtain
\begin{equation*}\dim_K(W) = \dim_L(W^L)=   \dim_L(V)f \geqslant 2^{\frac{e}{f}}f \geqslant 2 \frac{e}{f}f=2e.  \qedhere\end{equation*} 
\end{proof}

For a  group $A$, the largest normal $2$-subgroup of  $A$ is denoted $\OO_2(A)$.

\begin{prop}
\label{lem:2power}
Let $A$ be a group containing a non-abelian simple subgroup $G$. If $|A:G|$ is a power of $2$, then either
\begin{enumerate}
\item  $A/\OO_2(A)$ is almost simple with socle isomorphic to $G$, or 
\item $G\cong\Alt(2^n-1)$ and $A/\OO_2(A)$ is isomorphic to $\Alt(2^n)$ or $\Sym(2^n)$,  where $n\geq 3$.
\end{enumerate}

\end{prop}
\begin{proof}
 Note that $A/\OO_2(A)$ also satisfies the hypothesis hence we may assume that $\OO_2(A)=1$. Let $N$ be a minimal normal subgroup of $A$ and let $p$ be an odd prime that divides $|N|$. Since $|A:G|$ is a power of $2$, $G \cap N \neq 1$ and, since $G$ is simple, $G \leqslant N$.  Since distinct minimal normal subgroups intersect trivially, this shows that $N$ is the unique minimal normal subgroup of $A$. It also follows that $N$ is non-abelian hence $N=T_1 \times \dots\times T_k$ where  $T_i \cong T$ for some non-abelian simple group $T$. Since $T_1$ is a minimal normal subgroup of $N$,  and has order divisible by $p$, the same argument as above gives $G\leqslant T_1$. Since $|A:G|$ is a power of 2, it follows that $|N:T_1|=|T|^{k-1}$ is a power of $2$ hence $k=1$, $N$ is simple  and $A$ is almost simple. If $N=G$, then (1) holds. Otherwise, $G<N$ and~\cite[Theorem 1]{guralnick} implies that $N\cong\Alt(2^n)$ and $G\cong\Alt(2^n-1)$ with $n\geqslant 3$.
\end{proof}

\begin{thm}\label{thm:simpleNoFour}
Non-abelian simple groups with  no element of order four are CCA.
\end{thm}
\begin{proof}
Let $G$ be a non-abelian simple group without elements of order four, let $\Gamma$ be a  connected Cayley graph on $G$ and let $A$ be the colour-preserving group of automorphisms of $\Gamma$.  Let $N=\OO_2(A)$ and let $K$ be the kernel of the action of $A$ on the set of $N$-orbits. Since $G$ has  no element of order four, Lemma~\ref{coolnew} implies that for all $v\in \Gamma$ we have $K_v=1$ and hence $K=N$. By Lemma~\ref{VxStab2Group}, $|A:G|$ is a power of $2$. Since $G$ has no element of order four, Proposition~\ref{lem:2power} implies that $NG$ is normal in $A$. We claim that $G$ centralises $N$. 

Suppose otherwise, and note therefore that $G$ acts faithfully on $N$, and therefore on $N/\Phi(N)$ (where $\Phi(N)$ denotes the Frattini subgroup of $N$), and we may identify $G$ with a subgroup of $\Aut(N/\Phi(N))\cong \GL(d,2)$ for some $d\in \mathbb N$. Let $P$ be a Sylow $2$-subgroup of $G$. Since $K_v=1$ for all $v\in \Gamma$,  $|N|$ is the size of an $N$-orbit hence $|N|$ divides $|\Gamma|=|G|$ and thus $|N|$ divides $|P|$. Note that $P$ must be elementary abelian since $G$ has  no element of order four, and $G$ must appear in Theorem~\ref{WalterTheorem}. Suppose that $G$ is not isomorphic to $\PSL(2,2^n)$. Then $|N|\leq|P|\leq 8$ and so $d\leqslant 3$. However $G$ is not a subgroup of $\GL(3,2)$, a contradiction. We may thus assume that $G\cong\PSL(2,2^n)$ and $|N|\leq |P|= 2^n$ so $d\leqslant n$. By Lemma~\ref{lem:sl2mods},  the smallest faithful representation for $\PSL(2,2^n)$ over $\FF_2$ is of dimension $2n$, so $d\geqslant 2n$, a contradiction. 

We have shown that $G$ centralises $N$, hence $N G = N \times G$ and  $G$ is characteristic in $NG$ which is normal in $A$. It follows that $G$ is normal in $A$ and hence $\Gamma$ is CCA. This concludes the proof.
\end{proof}

\section{Simple groups with elements of order four}\label{sec:SimpleWith}

In this section we complete the proof of Theorem~\ref{thm:simplegrouptheorem}  by showing that simple groups with elements of order four are non-CCA. We use the Classification of Finite Simple Groups and simply consider each family of groups in turn (ignoring those that appear in Theorem~\ref{WalterTheorem}).


\subsection{Alternating groups}

The idea of the proof for the alternating groups is used for each simple group considered in the rest of this section. Let $G=\Alt(n)$. Since $\Alt(5)$ does not have an element of order four, we may assume that $n\geq 6$.
Let $t=(1~2)(3~4~5~6)$, $\tau=t^2$ and  $H = G_1 \cong \Alt(n-1)$. We claim that $(S_H(\tau),t,\tau)$ is a non-CCA triple of $G$. Since $H$ is maximal in $G$ and $t\notin H$, we have $G= \langle S_H(\tau), t\rangle$, so (Ai) holds. By Lemma~\ref{lem:stau}, (Aii) holds. By definition of $\tau$, (Aiii) holds (where we  take $T=\{t\}$). By definition, $S_H(\tau) \subseteq H$ and $\tau \in S_H(\tau)$, so $\langle S_H(\tau) \rangle \leqslant H $ and (Aiv)  holds (in fact, Remark~\ref{rem:invs} shows that $\langle S_H(\tau) \rangle =H$). Finally, (Av) is clear as $G$ has trivial centre.  Hence, Proposition~\ref{prop:gensets} shows that $G$ is non-CCA.

\begin{rem}

One can also prove that $\Sym(n)$ is non-CCA for $n\geq 5$. In fact, the same proof as above works for $n\geqslant 6$. For $n=5$ we take $t= (1~4~2~5) \text{, and } \tau = t^2$ and $H=\Sym(5)_{\{4,5\}}\cong \Sym(3)$. Then $(S_H(\tau),t,\tau)$ is a non-CCA triple of $G$.

By Theorem~\ref{thm:simpleNoFour}, $\Alt(5)$ is CCA. Moreover, $\Alt(n)$ is CCA for $n\leq 4$, whereas $\Sym(4)$ is not CCA, but $\Sym(3)$ and $\Sym(2)$ are CCA (see for example~\cite{FirstCCA}). One can also check that almost simple groups with socle $\Alt(6)$ are not CCA (using {\sc Magma} \cite{magma}, for example).

These results include every almost simple group whose socle is alternating. In each case, a group is CCA if and only if it does not contain an element of order four.
\end{rem}



\subsection{Chevalley groups}
 We now turn to the Chevalley groups (also called untwisted groups of Lie type).
Most of the families can be dealt with in a uniform manner, but to do this we require some setup. Our approach is to use the Chevalley presentation. We refer (and recommend) the reader to \cite{carter} or \cite{GLS3} for a more detailed exposition. In particular, all details in the following paragraphs are found in  \cite[Section 2.4]{GLS3}. First, we recall the notation. Let $G=X_n(q)$ be a \emph{Chevalley group} where $q=p^f$ for a prime $p$, $X \in \{\A,\B,\C,\D,\E,\F,\G\}$ and $n$ is a positive integer, with $n\geqslant 1$ if $X=\A$, $n\geqslant 2$ if $X=\B$, $n\geqslant 3$ if $X= \C$, $n\geqslant 4$ if $X = \D$, $n \in \{6,7,8\}$ if $X=\E$, $n=4$ if $X=\F$ and $n=2$ if $X=\G$. 
Associated to $G$ is a root system $\Phi$ (a set of vectors in a vector space associated to $G$)  with fundamental system $\Pi$ so that $\Phi = \Phi^+ \cup \Phi^-$ with respect to $\Pi$ (that is, each vector in $\Phi$ can be written as either a positive  or a negative linear combination of elements of $\Pi$). For explicit models of the root systems see \cite[Remark 1.8.8]{GLS3}. For each $\alpha\in \Phi$ we have  homomorphisms   $x_\alpha : (\FF_q,+) \rightarrow G$ and $h_\alpha : (\FF_q-\{0\},\times) \rightarrow G$. The \emph{root subgroups} of $G$ are $X_\alpha = \langle x_\alpha ( \eta ) \mid \eta \in \FF_q\rangle$. Finally, there are elements $n_\alpha \in G$ for each $\alpha \in \Phi^+$.

With this notation in hand, we set 
\begin{eqnarray*}
U & =&  \langle  x_\alpha(\eta) \mid \alpha \in \Phi^+, \eta \in \FF_q \rangle,\\
 H & =&   \langle  h_\alpha(\lambda) \mid \alpha \in \Phi^+, \lambda \in \FF_q - \{0\} \rangle = \langle  h_\alpha(\lambda) \mid \alpha \in \Pi, \lambda \in \FF_q - \{0\} \rangle  \\
 N & = &  \langle H , n_\alpha \mid \alpha \in \Phi^+ \rangle.
 \end{eqnarray*}
  Then $U$ is a Sylow $p$-subgroup of $G$, $H$ normalises $U$ and $B=UH$ is the normaliser in $G$ of $U$. Further, $H= B \cap N$  is abelian and $H$ is normal in $N$.  The Weyl group of $G$ is  $W :=N/H$,   with generators $s_\alpha:=Hn_\alpha$ such that $\alpha \in \Pi$. The Weyl group of $G$ acts faithfully on the root system $\Phi$, and we write $s_\alpha(\beta)$ for the image of $\beta \in \Phi$ under the Weyl group element $s_\alpha$.   The Chevalley Relations \cite[Theorem 2.4.8]{GLS3} give a presentation for the group $X_n(q)$ in terms of the elements of $U$, $H$ and $N$  as follows. For $\alpha,\beta \in \Phi$ with $\beta \neq \pm \alpha$, if $\alpha + \beta \notin \Phi$, then $[X_\alpha,X_\beta ] = 1$ and if $\alpha +\beta \in \Phi$, then the Chevalley Commutator Formula \cite[Theorem 2.4.5]{GLS3} allows us to calculate $[X_\alpha,X_\beta]$. The action of $N$ on the root subgroups is given by $$x_\beta(\eta )^{n_\alpha} = x_{s_\alpha(\beta)}(\pm \eta) \quad \text{and}\quad
x_\beta(\eta)^{h_\alpha(\lambda)} = x_\beta(\eta \lambda^{A_{\beta\alpha}})$$
where the constants $A_{\beta\alpha}$ are found in the Cartan matrix of $\Phi$, which we again do not dwell upon. Finally, the action of $N$ on $H$ is given by 
$ h_\beta( \lambda )^{n_\alpha} = h_{s_\alpha(\beta)}( \lambda)$ and we mention that $(n_\alpha)^2 = h_\alpha(-1)$ for all $\alpha \in \Phi^+$.


\subsubsection{\mathversion{bold}$X_n\notin\{\A_1,\G_2\}$}
We are now ready to deal with most Chevalley groups, except for two families of low rank.

Since the rank of $G$ ($=|\Pi|$) is at least two and since $G$ is not of type $\G_2$, we may pick simple roots $\alpha, \beta \in \Pi$ such that with $\gamma:=s_\alpha(\beta)$ we have $\beta + \gamma \notin \Phi$. More precisely, if  $G$ is not of type $\B_2$, the simple roots corresponding to nodes $1$ and $2$ of the Dynkin diagram of $G$ (as labelled in \cite[pg.12]{GLS3}) have this property. If $G$ is of type $\B_2$ we set $\alpha$ to be the short root, so that $s_\alpha(\beta)=2\alpha+\beta$. Set $J=\Pi-\{\alpha\}$ and let $P_J = \langle B, n_\gamma \mid \gamma \in J \rangle$ be the maximal parabolic  subgroup corresponding to $J$ ($P_J$ is a maximal subgroup of $G$ by \cite[(43.7)]{asch}).

Suppose first that $q$ is even. In this case, we have that  $\{n_\gamma \mid \gamma \in \Pi\}$ is a set of  involutions in $G$ and we let $t = x_\beta(1)n_\alpha$. Then 
$$\tau:=  t^2 = x_\beta(1) x_\beta(1)^{n_\alpha} = x_\beta(1) x_{s_{\alpha}(\beta)}(1)=x_\beta(1)x_\gamma(1).$$
Since $\beta + \gamma \notin \Phi$, $x_\beta(1)$ and $x_\gamma( 1)$ commute, and since $q$ is even, we have $\tau^2 = 1$. Note that $\tau \in P_J$ since $\tau \in U$, but $t \notin P_J$ since $n_\alpha \notin P_J$.  Let $S:= \langle S_{P_J}(\tau)\rangle$. We claim that $P_J$ is the unique maximal subgroup of $G$ containing $S$. Indeed, suppose that $M$ is a maximal subgroup of $G$ containing $S$. First note that since $q$ is even the elements $x_\delta(\eta)$ for $\eta \in \FF_q$ and $\delta \in \Phi$ are involutions, thus since $S$ contains each involution in $P_J$,  $S$ contains $U$. Hence by \cite[Theorem 2.6.7]{GLS3}, $M$ must be a parabolic  subgroup containing $B$. Since $S$ contains all of the involutions $n_\delta$ for $\delta \in J$, we have $Bn_\delta \subset M$. Hence $P_J=\langle B, n_\gamma \mid \gamma \in J \rangle \leqslant M$   which forces $P_J=M$ since $P_J$ is maximal. This proves the claim. Now it is easy to verify that $(S_{P_J}(\tau),t,\tau)$ is a non-CCA triple of $G$.

Suppose now that $q$ is odd. Set $t =n_\alpha$, so that (by \cite[Remark 2.4.0(c)]{GLS3} and \cite[Theorem 1.12.1k]{GLS3})  we have   $\tau := t^2 =h_\alpha(-1) \neq 1$ (note that when  $G$ is of type $\B_2$, this is due to our choice of $\alpha$). Let $S=\langle S_{P_J}(\tau)\rangle$. Let $\gamma \in \Pi$ be arbitrary and let $\eta \in \FF_q$. We have 
$$(x_\gamma(\eta))^\tau = x_\gamma((-1)^{A_{\gamma \alpha}}\eta) = x_\gamma(\pm \eta) = (x_\gamma(\eta))^{\pm 1}.$$
In particular, $\tau$ inverts or centralises each generator of $U$ and so  $U \leqslant   S $. If $n_\beta$ is an involution in $G$ (when $G$ is of type $\B_2$ for example) then $n_\beta \in S$. If  $n_\beta$ has order four, then $(n_\beta)^2 =h_\beta(-1)$ and   (recalling that $H$ is abelian)
\begin{eqnarray*}
(\tau n_\beta)^2 & =&  h_\alpha(-1)n_\beta^2 h_\alpha(-1)^{n_\beta} \\
& =&  h_\alpha(-1)h_\beta(-1) h_{s_\beta(\alpha)}(-1) \\
& =&  h_\alpha(-1)h_\beta(-1) h_{\alpha+\beta}(-1) \\
&=&  (h_\alpha(-1)h_\beta(-1) )^2 \\
& =& 1
\end{eqnarray*}
where the second to last equality holds by \cite[Theorem 2.4.7]{GLS3}. Hence $\tau n_\beta$ is an involution in $P_J$, and so $n_\beta \in S$ since $\tau\in S$. Next, if  $\gamma \in \Pi$ with $\alpha \neq \gamma  \neq \beta$, then $n_\gamma$ commutes with $h_\alpha(-1)=\tau$, so that $n_\gamma \in S$. Hence, arguing as above, if $M$ is a maximal subgroup of $G$ containing $S$, then $M=P_J$. In particular, $(S_{P_J}(\tau),t,\tau)$ is a non-CCA triple of $G$, so Proposition~\ref{prop:gensets} completes the proof  in this case.

We now turn to the excluded families.

\subsubsection{\mathversion{bold}$X_n=\A_1$}
We now assume that $G=\A_1(q)\cong \PSL(2,q)$. In view of Theorem~\ref{WalterTheorem}, we can assume that $q \equiv 1, 7 \pmod 8$ and, taking into account exceptional isomorphisms, that $q\notin \{7,9\}$.
In this case, $G$ has one conjugacy classes of involutions, so each involution in $G$ is a square. Further,  $G$ has maximal subgroups $M_1$ and $M_2$ isomorphic to $\D_{q+1}$ and $\D_{q-1}$. Depending on $q$, we may pick $i\in \{1,2\}$ and an involution $\tau$ in $M_i$ such that $\tau$ is a non-square in $M_i$  and $M_i = \langle S_{M_i}(\tau)\rangle$. Since $\tau$ is a square in $G$, but not in $M_i$, there is $t  \in G$ such that $t^2 = \tau$ and $t\notin M_i$. Hence $(S_{M_i}(\tau),t,\tau)$ is a non-CCA triple of $G$, and Proposition~\ref{prop:gensets} completes the proof  in this case.

\subsubsection{\mathversion{bold}$X_n=\G_2$}

Finally, we assume that $G=\G_2(q)$. We may assume that $q\geq 3$.
If $q$ is odd, then $G$ has a unique conjugacy class of involutions \cite[Theorem (4.4)]{chang}, so each involution in $G$ is a square. If $q$ is even, then $G$ has two conjugacy classes of involutions \cite[Proposition 2.6]{enomoto} and one of the classes consists of squares -- using notation of loc.~cit.~the  involution $x_3$ is the square of $x_5$, for example.  We consult \cite[Table 4.1]{Wilson} for the maximal subgroups of $G$. If $q$ is odd, 
let $M$ be a maximal subgroup   with $M \cong \SL(3,q) : 2$ and let  $\tau$ be  an involution in $M$ that does not lie in the  $\SL(3,q)$ subgroup. If $q$ is even, 
 let $M$ be a maximal subgroup with $M \cong \PSL(2,q)\times \PSL(2,q)$ and let $\tau$ be an involution in $M$ which is a conjugate of $x_3$ (the first paragraph of \cite[Section 4.3.6]{Wilson} shows that $M$ contains such an involution).   If $q$ is even, the structure of $M$ shows that $\tau$ is not a square in $M$   and, if $q$ is even, then the Sylow 2-subgroups of $M$ are elementary abelian, so $\tau$ is not a square in $M$. Thus, in either case, there is $t\in G$ such that $t^2 = \tau$. Now $t\notin M$ so  $\langle M , t \rangle = G$. 
In both cases, we have $ \langle S_M(\tau) \rangle = M$.   Thus $(S_M(\tau),t,\tau)$ is a non-CCA triple of $G$, so Proposition~\ref{prop:gensets} completes the proof  in this case.

\subsection{Twisted groups of Lie type}

Next, we deal with the so-called twisted (or ``Steinberg") groups. These groups arise as fixed points of the so-called \emph{graph-field automorphisms} of the Chevalley groups, which exist whenever the associated Dynkin diagrams admit a graph automorphism (of order $d$ if the group is ${}^dX_n(q)$).

We start with the case $d=2$ that is, groups of the form  ${}^2\A_n(q) \cong \PSU(n+1,q)$, ${}^2 \D_n(q) \cong \mathrm P \Omega^-(2n,q)$ ($n\geqslant 4$) and ${}^2 \E_6(q)$. 
Below we have shown the Dynkin diagrams in such a way that the orbits of the graph automorphism are clear.

\begin{figure}[h]
\includegraphics[scale=0.6]{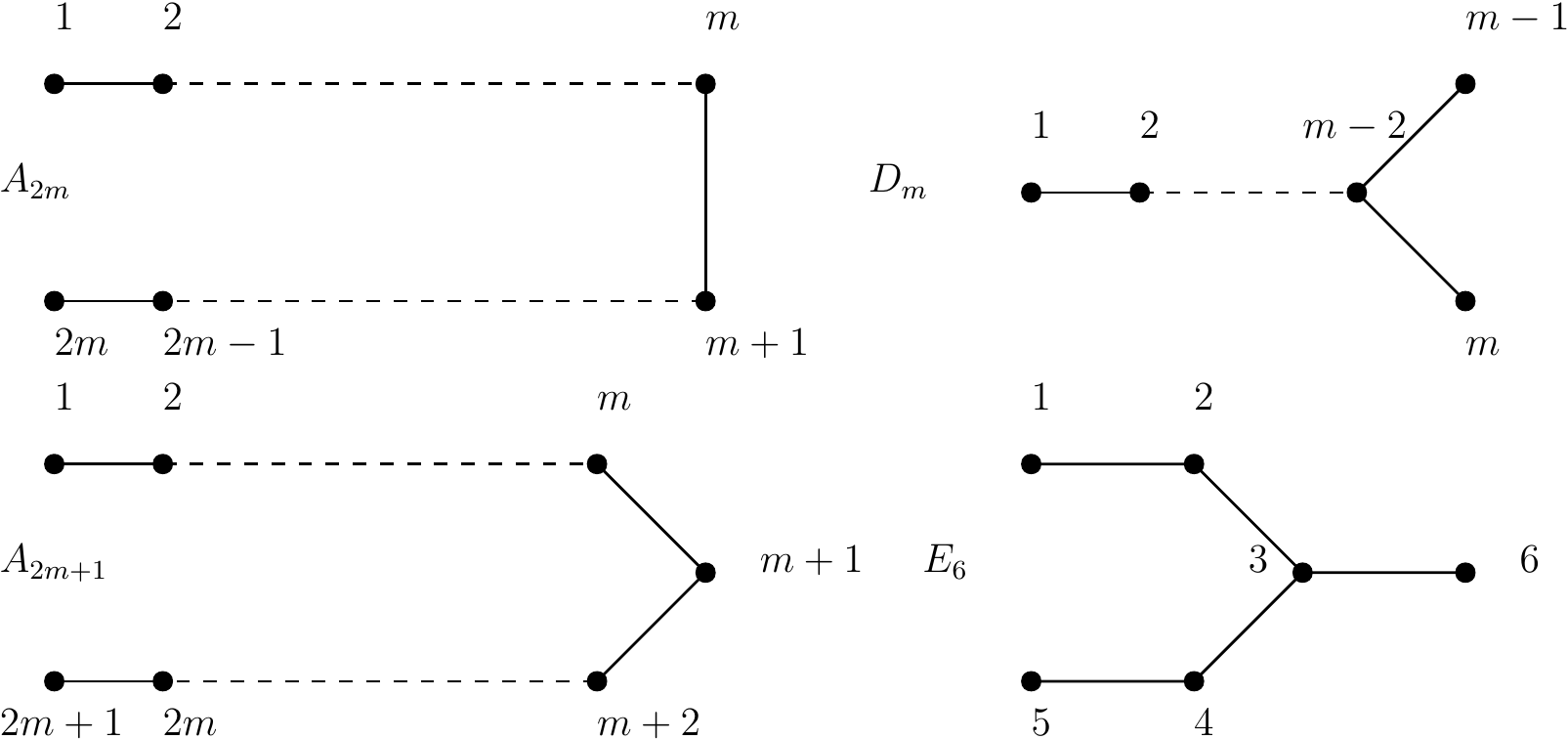}
\end{figure}

A twisted group ${}^2X_n(q)$ admits a \emph{twisted} root system. These root systems are in fact equivalence classes of the images of roots from the corresponding untwisted group $X_n(q)$, see \cite[Definition 2.3.1]{GLS3}.  The graph automorphism of the Dynkin diagram extends to an isometry $\rho$ of the vector space associated to the untwisted group $X_n(q)$, and thus acts on the root system $\Phi$. The average of the roots in the orbit of $\alpha\in \Phi$  is denoted $\tilde{\alpha}$, and the set of these averages is denoted $\tilde{\Phi}$. We say two vectors of $\tilde{\Phi}$ are equivalent if  one is a  positive multiple of the other. The equivalence class of  $\tilde{\alpha}$ is denoted by $\hat{\alpha}$. Thus there are maps $\Phi \rightarrow \tilde{\Phi} \rightarrow \hat{\Phi}$, and the  set $\tilde{\Phi}$ is the twisted root system of the twisted group ${}^2X_n(q)$ (it may not be an actual root system). In the cases under consideration, $\tilde{\Sigma}=\hat{\Sigma}$, except for ${}^2 A_{2m}(q)$. The twisted group is generated by the root subgroups $\langle x_{\hat{\alpha}}(\eta) \mid \eta \in \FF_q \rangle$ for $\hat{\alpha} \in \hat{\Phi}$. The  definition of $x_{\hat{\alpha}}$ depends upon the orbit of $\alpha$ under $\rho$, and is found in \cite[Table 2.4]{GLS3}. Three possibilities arise for us, the orbit is a single vertex,  two disconnected vertices or  two vertices joined by an edge. The definition of $x_{\hat{\alpha}}$ is then found in Row I,  II or IV  respectively of \cite[Table 2.4]{GLS3}. (For example, in the case of ${}^2\D_4(q)$, we have $\alpha_3$ and $\alpha_4$ are in the same orbit, and $x_{\hat{\alpha_3}}(\eta) = x_{\alpha_3}(\eta)x_{\alpha_4}(\eta^q)$, see Row II of \cite[Table 2.4]{GLS3}.) There is an analogous definition of $h_{\hat{\alpha}}$ for $\hat{\alpha} \in \hat{\Phi}$, found in \cite[Table 2.4.7]{GLS3}. For the precise definition of $n_{\hat{\alpha}}$ we actually consult the proof of \cite[Theorem 2.4.8]{GLS3} (see also \cite[Remark 2.4.9(b)]{GLS3}) and conclude that either $\alpha$ is fixed by the isometry $\rho$ and we can take $n_{\hat{\alpha}}=n_\alpha$ or the orbit of $\alpha$ has length two. In even characteristic, $n_{\hat{\alpha}}$ always has order two and in odd characteristic $n_{\hat{\alpha}}$ has order four if ${\hat{\alpha}}$ is of type II, and has order two if $\hat{\alpha}$ is of type IV.  More specifically, when $\hat{\alpha}$ is of type II, we have  $n_{\hat{\alpha}}=h n_\alpha n_{\alpha^\rho}$ for some $h\in T_1$ (see loc.~cit.), where $h=1$ in even characteristic. If $\hat{\alpha}$ is of type IV, we have   $n _{\hat{\alpha}} = h n_{\alpha}n_{\alpha^\rho}n_{\alpha}=h n_{\alpha^\rho}n_{\alpha}n_{\alpha^\rho}$ for some $h\in T_1$ (see loc.~cit.), where again $h=1$ if the characteristic is even. We calculate that, if $n_{\hat{\alpha}}$ has order four, then (in all cases) $n_{\hat{\alpha}}^2 = h_{\hat{\alpha}}(-1)$.
 
\subsubsection{\textbf{\mathversion{bold}$^2\D_n(q)$ for $n\geq 4$ and $^2\E_6(q)$}}
Let $G={}^2X_n(q)$, where $X \in \{D,E\}$.
The idea of the proof is analogous to the untwisted case, but the setup requires a little more care because of the twist. First, we select two simple roots $\alpha, \beta \in \Pi$  that are fixed by the isometry $\rho$. In detail: if $X=D$ pick $\alpha=\alpha_1$, $\beta=\alpha_2$ and if $X=E$ pick $\alpha=\alpha_4$ and $\beta = \alpha_3$. The proof can now proceed exactly as for the untwisted groups. We provide some details to show  how the various steps may be adjusted. Let $\delta \in \{\alpha,\beta\}$. Since $\delta$ is fixed by the isometry $\rho$, any calculations involving $n_{\hat{\delta}}=n_{\delta}$ and $h_{\hat{\delta}}(\eta)=h_\delta(\eta)$ are the same as for the untwisted groups. Further, the calculations regarding $s_\alpha(\beta)$ and $\beta$ are the same as for the untwisted group. Finally, since all of the root subgroups are either of type I or II, the assertions concerning the structure of the Sylow $p$-subgroups (that it is generated by involutions when $q$ is even and that the generators are either inverted or centralised when $q$ is odd) are the same as for the untwisted groups. 

\subsubsection{\textbf{\mathversion{bold}$^2\A_n(q)$ for $n\geq 3$}}
 Let $n\geq 3$ and let $G={}^2\A_{n}(q)\cong \PSU(n+1,q)$.
For the unitary groups in odd dimension, the root subgroups can be of type IV and so, in even characteristic, are not necessarily generated by involutions. For this reason,  we use a different approach.
We first claim that there is a maximal parabolic subgroup $P$ stabilising a totally isotropic $1$-space or $2$-space and an element $t \in G - P$ of order four such that $\tau:=t^2 \in P$.

Suppose first that $q$ is odd. Since $n\geqslant 3$, the claim holds with $t=n_{\hat{\alpha}_1}$, $\tau=t^2 = h_{\hat{\alpha}_1}(-1)$ and $P$ the stabiliser of a totally isotropic  1-space corresponding to the first node of the Dynkin diagram for $A_n(q)$.

Suppose now that $q$ is even and assume first that $n \geqslant 4$. Set $w= n_{\hat{\alpha}_2}$, $x= x_{\hat{\alpha}_1}(1)$ and put $t=xw$. Since $q$ is even, $n_{\hat{\alpha}_2}$ is an involution and,  since  $n\geqslant 3$, the twisted root $\hat{\alpha}_1$ is of type II, so that both $x$ and $w$ are involutions. We set $\tau:=t^2   = xx^w$. To calculate $\tau$, there are three cases depending on the type of $\hat{\alpha_2}$. Since we assume $n\geqslant 4$, $\hat{\alpha_2}$ is of type II or IV.

If $\hat{\alpha_2}$ is of type II then necessarily $n \geqslant 5$ and  $w=n_{\alpha_2}n_{\alpha_{n-1}}$. Then
\begin{eqnarray*}
xx^w =& x_{\alpha_1}(1) x_{\alpha_{n}}(1) (x_{\alpha_1}(1))^{n_{\alpha_2}}(x_{\alpha_{n}}(1))^{n_{\alpha_{n-1}}}\\
= &   x_{\alpha_1+\alpha_2}(1) x_{\alpha_2}(1) x_{\alpha_{n}}(1)x_{\alpha_{n}+\alpha_{n-1}}(1).
\end{eqnarray*}
Note that since $n\geqslant 5$, each of the elements in the expression for $xx^w$ above commute, and so $\tau^2=1$ since $q$ is even. Hence the claim holds with $P$ the stabiliser of a totally isotropic $2$-space.

Suppose that $\hat{\alpha_2}$ is of type IV, so that $n=4$, $x=x_1(1)x_4(1)$ and $n_{\hat{\alpha_2}} = n_{\alpha_2}n_{\alpha_3}n_{\alpha_2}= n_{\alpha_3}n_{\alpha_2}n_{\alpha_3}$. Then 
\begin{eqnarray*}
xx^w = &   x_{\alpha_1}(1)    x_{\alpha_4}(1)     ( x_{\alpha_1}(1))^{n_{\alpha_2}n_{\alpha_3}}
(x_{\alpha_4}(1))^{n_{\alpha_2} n_{\alpha_3}}\\
= &   x_{\alpha_1}(1)    x_{\alpha_4}(1)   x_{\alpha_1+\alpha_2+\alpha_3}(1) x_{\alpha_2+\alpha_3+\alpha_4}(1).
\end{eqnarray*}
Using the Chevalley Commutator Formula,  we calculate that   $\tau^2 =1$. (Note that there are exactly two pairs of non-commuting elements in the expression for $\tau$, so  when calculating the square, we introduce twice the element $x_{\alpha_1+\alpha_2+\alpha_3+\alpha_4}(1)$ which is the commutator of both $x_{\alpha_1}(1)$ and $x_{\alpha_2+\alpha_3+\alpha_4}(1)$ and of $x_{\alpha_4}(1)$ and $x_{\alpha_1+\alpha_2+\alpha_3}(1)$.) Hence the claim holds with $P$ the stabiliser of a totally isotropic $2$-space in the 5-dimensional vector space on which $\SU(5,q)$ acts.

Now assume that $n=3$ and set $w=n_{\hat{\alpha_1}}$,  $x=x_{\hat{\alpha_2}}(1)=x_{\alpha_2}(1)$ and $t=xw$. Since $q$ is even, $w$ is an involution, so $\tau:=t^2 = xx^w$. Since $w = n_{\alpha_1}n_{ \alpha_3}$, we calculate that $\tau=x_{\alpha_2}(1)x_{\alpha_1+\alpha_2+\alpha_3}(1)$. Hence $\tau^2 =1$ and the claim holds with $P$ the stabiliser of a totally isotropic $1$-space in the 4-dimensional vector space on which $\SU(4,q)$ acts.

Set $S=\langle S_P(\tau) \rangle$ and let $X$ be the normal subgroup of $P$ generated by the involutions in $P$, so that $X \leqslant S$. We claim that $P$ is the unique maximal subgroup of $G$ containing $P$. The proof of this claim depends on the structure of $P$, for which we refer to \cite[Theorem 3.9(ii)]{Wilson} (noting that  $n-k$ should read $n-2k$). Let $P=QL$ be the Levi decomposition of $P$.

Let $M$ be a normal quasisimple subgroup of $L$, so that $M/Z(M)$ is non-abelian simple. We claim that  $M \leqslant S$.  Since $M$ is quasisimple, either $X \cap M \leqslant Z(M)$ or $X\cap M = M \leqslant S$. If the former holds, then $[X, M, M] =1 = [M,X,M]$ and so the Three Subgroups Lemma   gives $1=[M,M,X]=[M,X]$, where the last equality holds since $M$ is perfect. Since $\tau \in X$, this yields $M \leqslant C_P(\tau)$ and so $M \leqslant S$. Let $E$ be the product of the normal quasisimple subgroups of $L$, so that $E \leqslant S$.

Assume that $(n,q) \notin \{(3,2),(3,3),(5,2),(6,2)\}$. Then $QE$ contains a Sylow $p$-subgroup of $G$. Note that $ Q = [Q,E ] \leqslant [Q,X] \leqslant X$. Hence $QE \leqslant S$ and so $S$ contains a Sylow $p$-subgroup of $G$. By \cite[Theorem 2.6.7]{GLS3} the only maximal subgroup containing $S$ must therefore be a parabolic  subgroup, that is, a stabiliser of some totally isotropic subspace. Since $QE$ fixes a unique totally isotropic  subspace, 
the only maximal subgroup  of $G$ containing $S$ is $P$, as claimed.

For $(n,q) \in \{(3,2),(3,3),(5,2),(6,2)\}$, one can verify the claim directly (say, with {\sc Magma} \cite{magma}).

From the claim, it follows $(S_P(\tau),t,\tau)$ is a non-CCA triple of $G$, and so $G$ is non-CCA by Proposition~\ref{prop:gensets}.

\subsubsection{\mathversion{bold}$^2\A_2(q)$}
Let $G ={}^2\A_2(q) \cong \PSU(3,q)$ for $q$ a prime power.
Let $M \leqslant \SU(3,q)$ be the stabiliser of a non-degenerate direct sum decomposition of the natural vector space that $\SU(3,q)$ acts on. Then $M \cong (\C_{q+1})^2 \rtimes \Sym(3)$ and $M$ is maximal in $\SU(3,q)$ by \cite{colva}. Since $\Z(\SU(3,q)) \cong \C_{(3,q+1)}$ is contained in $M$, we have $H:=M/Z(\SU(3,q))$ is maximal in  $G$.  Pick $\tau$ to be an involution in a subgroup of $H$ conjugate to $\Sym(3)$. Then $\tau$ is a non-square in $H$. Since $G$ has elements of order four, and one conjugacy class of involutions (see \cite[Table 4.5.1]{GLS3} for $q$ odd and \cite[(6.1)]{aschseitz} for $q$ even) there is $t \in G-H$ such that $t^2=\tau$. 

Let $X:=\langle S_H(\tau) \rangle$. We claim that $H =X$. Note that $X$  is a normal subgroup of $H$ containing our chosen $\Sym(3)$ subgroup. Write $q+1=p_1^{a_1}\dots p_r^{a_r}$. Then,  for $p_i\neq 3$, we have $T_i=\mathrm O_{p_i}(H)$ and $\Sym(3)$ acts irreducibly on $T_i \cong (C_{p_i^{a_i}})^2$. This forces $T_i \leqslant X$. For $p_i=3$, $T_i \cong \C_{3^{a_i}} \circ_{3} \C_{3^{a_i}}$ (where the symbol $\circ_3$ means that a subgroup of order three has been identified). It follows that there are generators $x$ and $y$ for $T_i$ such that $x^\tau=x$ and $y^\tau = y^{-1}$. Thus $y=z\tau$ for some involution $z$ and so $T_i \leqslant X$. 

The previous two paragraphs show that  $(S_M(\tau),t,\tau)$ is a non-CCA triple of $G$ hence, by Proposition~\ref{prop:gensets}, $G$ is non-CCA.

\subsubsection{\mathversion{bold}$^3\D_4(q)$}

Let $G={}^3\D_4(q)$ for $q$ a prime power.
When $q$ is odd, $G$ has a unique conjugacy class of involutions by \cite[Lemma 2.3(i)]{Kl}, so  each involution is a square in $G$. When $q$ is even,  $G$ has two conjugacy classes of involutions \cite[(8.1)]{thomas}. Representatives are  $x_6(1)$ and $x_4(1)$ (using the notation of \cite{thomas}). Using \cite[(2.1)]{thomas}, we calculate that $x_6(1) = (x_2(1)x_5(1))^2$. For $\alpha \in \FF_{q^3}-\FF_q$, we find that the square of $x_1(\alpha)x_3(1)$ is conjugate to $x_4(1)$ by \cite[(3.1)]{thomas}. Hence, for all $q$, every involution in $G$ is a square. Now, by \cite[Theorem 4.3]{Wilson}, there is a maximal subgroup $M$  of the form $2 \cdot(\PSL(2,q^3) \times \PSL(2,q)): 2$ for $q$ odd and, for $q$ even, there is a maximal subgroup $M$ of the form $\PSL_2(q^3) \times \PSL_2(q)$. For $q$ odd, we let $\tau$ be an involution in $M$ outside the derived subgroup and, for $q$ even, we let $\tau$ be  an involution in $M$. In both cases, we have that $\tau$ is a square in $G$, but not in $M$, so there is $t\in G$ such that $t^2 =\tau$.  This case can now be finished by arguing as in the previous one.


\subsection{Suzuki-Ree groups}
\subsubsection{\mathversion{bold}${}^2\F_4(q)$}
Let $G$ be either a large Ree group ${}^2\mathrm F_4(q)$,  where $q=2^e\geqslant 4$ with $e$ odd, or the Tits group ${}^2 \F_4(2)'$.
We note that $G$ has two conjugacy classes of involutions by \cite[(18.2)]{aschseitz}. Consulting \cite{Parrott} and using notation of loc.~cit., we find representatives $x_{12}(1)$ and $x_{10}(1)$. We find $x_{12}(1) = x_5(1)^2 $ and $(x_4(1)x_2(1))^2=x_7(1)x_8(1)x_{11}(1)x_{12}(1)$. The latter element is conjugate to $x_7(1)$ by \cite[Section 2]{Parrott} and $x_7(1)$ is conjugate to $x_{10}(1)$ by \cite[Lemma 10]{Parrott}. Thus all involutions in $G$ are squares. By \cite[Theorem 4.5]{Wilson} there is a maximal subgroup $M \cong \Sp(4,q):2$. Let $\tau$ be an involution in $M$ that does not lie in the $\Sp(4,q)$ subgroup. Then $M= \langle S_M(\tau) \rangle$ since $M$ is almost simple (even for $q=2$). There is $t \in G$ such that $t^2 = \tau$ and $t\notin M$ since the structure of $M$ shows that $\tau$ is not a square in $M$. Then $(S_M(\tau),t,\tau)$ is a non-CCA triple of $G$, and Proposition~\ref{prop:gensets} completes the proof.

\subsubsection{\mathversion{bold}${}^2\B_2(q)$}

 Let $q=2^e \geqslant 8$ with $e $ odd. We will prove something slightly stronger than required, namely that, if $G$ is an almost simple group with socle ${}^2\B_2(q)$, then $G$ is non-CCA.
By \cite[Theorem 7.3.5]{colva} and \cite[Table 8.16]{colva}, the normaliser in $G$ of a maximal subgroup $M \cong \D_{q-1}$ of $\mathrm{soc}(G)$ is maximal in $G$. Now $\N_G(M) \cong C_{q-1} \rtimes (C_2 \times C_f)$ for some divisor $f$ of $e$.  Since $q$ is even, we may pick involutions $x$ and $\tau$ and an element $c \in G$ of order $f$ that commutes with $\tau$  such that $\N_G(M)= \langle \tau x, \tau, c \rangle$. Since all involutions in $G$ are squares (see the description of a Sylow $2$-subgroup of $G$ in \cite[Table 8.16]{colva}), there is $t\in G$ such that $t^2 = \tau$. Now $(\{\tau x, \tau,c \},\{t\},\tau)$ is a non-CCA triple, so Proposition~\ref{prop:gensets} completes the proof.



\subsection{Sporadic groups}
 To finish the proof of Theorem~\ref{thm:simplegrouptheorem}, only the  sporadic groups remain to be addressed. In view of Theorem~\ref{WalterTheorem}, we can ignore the Janko group $\mathrm J_1$.
Let $(G,H)$ be one of the pairs from the table below. That $H$ is a maximal subgroup of $G$ and other facts regarding properties of $G$ used below can be found in  \cite{atlas} or \cite{onlineatlas}.
\begin{center}
\begin{tabular}{ c c c }
\begin{tabular}{|c || c |}
\hline
G & H   \\ \hline \hline
$\mathrm M_{11}$ & $\PSL(2,11)$   \\
$\mathrm M_{22}$ & $\PSL(2,11)$ \\
$\mathrm M_{23}$  &  $\mathrm M_{11} $ \\
$\mathrm M_{24}$  &  $\PSL(2,7)$ \\
$\mathrm J_3$  &  $\PSL(2,19)$ \\ 
$\mathrm{McL}$  & $\mathrm M_{11} $ \\
$\mathrm{He} $ & $ 7^{1+2} : (3 \times \Sym(3)) $ \\
$\mathrm{Suz} $ & $ \Alt(7) $ \\
$\mathrm{Co}_2 $ & $ \mathrm M_{23} $ \\ 
\hline
\end{tabular}
&
\begin{tabular}{|c || c |}
\hline
G & H   \\ \hline \hline
$\mathrm{Co}_1 $ & $ \mathrm{Co}_2 $ \\
$\mathrm{HN}$  &  $\Alt(12) $ \\
$\mathrm{ON}$ & $\Alt(7) $ \\
$ \mathrm J_4$ & $ \mathrm \PSU(3,3) $ \\ 
$ \mathrm {Ly}$ & $37:18 $ \\
$ \mathrm{Th}$ & $\PSL(3,3) $ \\
$ \mathrm{Fi}'_{24}$ & $ 29:14 $ \\
$ \mathrm M$ & $ \PSL(2,59) $ \\
\hline
\end{tabular}
&
\begin{tabular}{|c || c |}
\hline
G & H   \\ \hline \hline
$\mathrm M_{12}$ & $\mathrm M_{11} $   \\
$\mathrm J_2$  &  $\PSU(3,3)$ \\
$\mathrm{HS}$  & $\mathrm M_{11} $ \\ 
$\mathrm{Ru} $ & $ \Alt(8) $ \\
$\mathrm{Co}_3 $ & $ \mathrm M_{23} $ \\
$ \mathrm{Fi}_{22}$ & $\mathrm M_{12} $  \\
$ \mathrm{Fi}_{23}$ & $\PSL(2,23) $  \\
$ \mathrm B$ & $\mathrm M_{11} $ \\
\hline
\end{tabular}
\end{tabular}
\end{center}
 For $X=H$ or $X=G$ and $\tau$ an involution of $X$ we define 
$$_X\sqrt{\tau}=\{ t \in X \mid t^2 = \tau \}$$
and we have  that 
$$|_X\sqrt{\tau}|=\sum_{\chi \in \mathrm{Irr}(X)}s(\chi)\chi(\tau)$$
where $s(\chi)$ is the Frobenius-Schur indicator of $\chi$. For both $G$ and $H$, the character tables are either stored in {\sc Gap} \cite{gap} or can be computed easily. We used {\sc Gap} to compute the values  of $|_X\sqrt{\tau}|$ for each conjugacy class of involutions in both $G$ and $H$. For $G$ in the left or middle table above we let $\tau$ be  an involution of $H$.  For $G$ in the right  table, let $\tau \in H$ be an involution with  $|_H\sqrt{\tau}| \neq 0$.   Our calculations then show that $|_G\sqrt{\tau}| \neq |_H\sqrt{\tau}|$ for the pair $(G,H)$.  Hence there is $t\in G - H$ such that $t^2 = \tau$. Since $\langle S_H(\tau) \rangle$ contains both $C_H(\tau)$ and the normal subgroup of $H$ generated by the involutions of $H$, we have $\langle S_H(\tau) \rangle = H$. Then it is easy to check that $(S_H(\tau),\{t\},\tau)$ is a non-CCA triple of $G$, and Proposition~\ref{prop:gensets} completes the proof. 


\bigskip

 \noindent\textsc{Acknowledgements.}
We would like to thank Michael Giudici for his helpful comments on an early version of this manuscript.

\end{document}